\documentclass[11pt,a4paper]{article}
\usepackage{amsmath}
\usepackage{amssymb}
\usepackage{amsthm}
\usepackage{amsbsy}
\usepackage{mathrsfs}
\usepackage{tipa}
\usepackage{txfonts}
\usepackage{amsfonts}
\usepackage{graphicx}
\usepackage[all]{xy}
\usepackage{enumitem}
\usepackage{mathtools}
\usepackage{tikz}
\setenumerate[1]{itemsep=1pt,partopsep=0pt,parsep=\parskip,topsep=6pt}
\setitemize[1]{itemsep=1pt,partopsep=0pt,parsep=\parskip,topsep=6pt}
\setdescription{itemsep=1pt,partopsep=0pt,parsep=\parskip,topsep=6pt}
\input xypic
\newtheorem{thm}{Theorem}[section]
\newtheorem{cor}[thm]{Corollary}
\newtheorem{lem}[thm]{Lemma}

\newtheorem{que}[thm]{Question}
\newtheorem{exm}[thm]{Example}
\newtheorem{prop}[thm]{Proposition}
\newtheorem{defn}[thm]{Definition}
\newtheorem{rem}[thm]{Remark}
\newtheorem{nota}[thm]{Notation}
\usetikzlibrary{arrows.meta, calc, tikzmark}
\begin{document}

\title{\Large \bf  Virtually Gorenstein Artin algebras are weakly Gorenstein
\footnotetext{
*Corresponding author.  E-mail address: sdwg001@163.com.\\
This research was supported by the Natural Science Foundation of Hunan Province (No. 2023JJ30561).
 }}

 \vskip 0.5cm
\author{{\normalsize  Weiqing Li }\\
{\small College of Mathematics and Statistics, Jishou University, Jishou 416000,  P. R. China}\date{}}
 \maketitle
 \vskip-1cm

\begin{abstract}  We prove that  all (possibly  infinitely generated) semi-Gorenstein-projective modules over a virtually Gorenstein Artin algebra  are Gorenstein projective.
We also establish a new criterion for determining  the projective (injective) dimensions    of modules over virtually
Gorenstein Artin algebras.
This enables us  to show  the  validities of   the Auslander-Gorenstein Conjecture, the (Strong) Nakayama Conjecture and Tachikawa's First Conjecture for virtually
Gorenstein Artin algebras.  Finally, we prove that  the    Auslander-Reiten  Conjecture holds for virtually Gorenstein  Artin algebras  of finite CM-type.
\end{abstract}

\def\s{\stackrel}
\def\Longrightarrow{{\longrightarrow}}
\def\A{\mathcal{A}}
\def\B{\mathcal{B}}
\def\C{\mathcal{C}}
\def\D{\mathcal{D}}
\def\T{\mathcal{T}}
\def\R{\mathcal{R}}
\def\P{\mathcal{P}}
\def\S{\mathcal{S}}
\def\H{\mathcal{H}}
\def\U{\mathscr{U}}
\def\V{\mathscr{V}}
\def\M{\mathscr{M}}
\def\N{\mathcal{N}}
\def\W{\mathscr{W}}
\def\X{\mathscr{X}}
\def\Y{\mathscr{Y}}
\def\Z{\mathcal {Z}}
\def\I{\mathcal {I}}
\def\add{\mbox{add}}
\def\Aut{\mbox{Aut}}
\def\coker{\mbox{coker}}
\def\deg{\mbox{deg}}
\def\diag{\mbox{diag}}
\def\dim{\mbox{dim}}
\def\End{\mbox{End}}
\def\Ext{\mbox{Ext}}
\def\Hom{\mbox{Hom}}
\def\Gr{\mbox{Gr}}
\def\id{\mbox{id}}
\def\Im{\mbox{Im}}
\def\ind{\mbox{ind}}
\def\mod{\mbox{mod}}
\def\mul{\multiput}
\def\c{\circ}
\def \text{\mbox}

\hyphenation{ap-pro-xi-ma-tion}


\medskip


\section{ Introduction}

Throughout this paper, $\Lambda$ is an Artin algebra and
all $\Lambda$-modules are unitary  right  $\Lambda$-modules. We view left $\Lambda$-modules as right  ones over
the opposite ring $\Lambda^{{\rm op}}$. Let
 Mod-$\Lambda$ and mod-$\Lambda$
  denote the categories of all right
 $\Lambda$-modules and all finitely generated right
$\Lambda$-modules, respectively.
The symbols  id$_{\Lambda}$$(M)$ and pd$_{\Lambda}$$(M)$  stand  for the   injective and projective   dimension of
  a  right $\Lambda$-module $M$, respectively.

The class of Gorenstein projective  modules plays a central role in the theory of Gorenstein homological algebra.

\begin{defn} \label{defn1} \rm{  (\cite{eno9}).
 A  {\it  complete projective resolution} is an exact sequence of projective
 $\Lambda$-modules,  $$\cdots  \rightarrow P_{1} \rightarrow P_{0}\rightarrow
P_{-1}\rightarrow P_{-2}\rightarrow \cdots,  $$
 such that Hom$_{\Lambda}$($-$, $Q$) leaves the sequence exact whenever $Q$ is a projective   $\Lambda$-module; the module $M$ = im($P_{0}\rightarrow
P_{-1}$) is then  said to be {\it Gorenstein  projective}.

{\it Gorenstein  injective }  modules are defined dually.
 We use $\mathscr{GP}$ and $\mathscr{GI}$    to denote, respectively,  the classes of all Gorenstein  projective and  injective
   (right)  $\Lambda$-modules. }\end{defn}

\begin{defn} \label{defn100}  \rm{  (\cite{hol}). Let $M$ be a    $\Lambda$-module.  We say that  $M$ has {\it Gorenstein projective
dimension} at most  $n$, and we write  Gpd$_{\Lambda}$($M$) $\leq n$, if there exists an exact sequence
of  $\Lambda$-modules $ 0 \rightarrow G_n\rightarrow
\cdots \rightarrow G_0\rightarrow M\rightarrow 0$ where each $G_{i}$ is Gorenstein projective. If there is no such $n$, set   Gpd$_{\Lambda}$($M$)
$=\infty$.

The {\it Gorenstein injective
dimension},  Gid$_{\Lambda}$($M$), is     defined similarly.}\end{defn}

\begin{defn} \label{defn81} \rm{   (\cite{rz1}).
 A right $\Lambda$-module
$M$   is called   {\it semi-Gorenstein-projective} if $M\in$ $^{\perp_{\infty}}{\Lambda}$, i.e.,    Ext$ _{\Lambda}^{i}(M, \Lambda)=0$  for all  $i \geq 1$.
 }\end{defn}

It is clear that $\mathscr{GP} \subseteq$  $^{\perp_{\infty}}{\Lambda}$; but we learn from   \cite{rz1}
that   there exist  semi-Gorenstein-projective modules which are not Gorenstein projective.

\begin{defn} \label{defn82} \rm{   (\cite{rz1}).
 An Artin algebra $\Lambda$  is said to be   {\it right  weakly Gorenstein} provided that $^{\perp_{\infty}}{\Lambda} $ $\cap$ mod-$\Lambda \subseteq$ $\mathscr{GP}$.
 If both  $\Lambda$ and $\Lambda^{{\rm op}}$ are right  weakly Gorenstein,   then $\Lambda$  is called {\it   weakly Gorenstein}. }\end{defn}

Recall that an artin algebra  $\Lambda$  is said to be of { \it finite CM-type} \cite{bel1}  if there are only finitely many, up to isomorphisms,
indecomposable finitely generated Gorenstein projective modules.
In 2020, Ringel and  Zhang \cite{rz1} posed the following problem.

\begin{que} \label{que1} { \rm   (\cite[Question 9.2]{rz1})}. If $\Lambda$ is  an Artin algebra of finite CM-type, is $\Lambda$  right  weakly Gorenstein?
There is a weaker question:
is  an Artin algebra $\Lambda$  right  weakly Gorenstein in case all  finitely generated right Gorenstein  projective modules are   projective?\end{que}

The first main purpose of this paper is to restrict our attention to the aforementioned question over virtually Gorenstein  Artin algebras, which is a natural generalization of the class of Gorenstein algebras.

\begin{defn}\label{defn21} { \rm   (\cite[Definition 8.1]{ber})}.   {\rm  An  Artin algebra $\Lambda$  is said to be}   virtually Gorenstein   {\rm provided that}   $\mathscr{GP}^\perp$ = $^\perp\mathscr{GI}$.
  \end{defn}

It is well-known that   an  Artin algebra $\Lambda$ is virtually Gorenstein  if and only if $\Lambda^{{\rm op}}$ is  virtually Gorenstein (see \cite[Theorem 8.7]{bel}).
In this note,  we prove the following theorem (see Theorem \ref{thm101}), which is closely related to
  Question \ref{que1}.
 \begin{thm}\label{thm81} Let   $\Lambda$  be  a virtually Gorenstein  Artin algebra. Then  $^{\perp_{\infty}}{\Lambda} $ $=$  $\mathscr{GP}$.
 In particular,  $\Lambda$ is weakly Gorenstein.
 \end{thm}

By virtue of Theorem  \ref{thm81} and Example \ref{exm1}, the class of virtually Gorenstein algebras forms a proper subclass within the class of weakly Gorenstein algebras.

Let
 $$ 0  \longrightarrow \Lambda_{\Lambda}
\longrightarrow
E_0 \longrightarrow E_1
\longrightarrow \cdots \longrightarrow E_i
\longrightarrow \cdots   $$ be a minimal  injective resolution of $\Lambda_{\Lambda}$.

In 1958,   Nakayama \cite{nak}   posed a conjecture, which by results
of M$\ddot{\rm u}$eller \cite{mue}   is equivalent to the following:
\\
\\
{\bf Nakayama Conjecture (NC) }\ \ {\it  Let   $\Lambda$  be    an Artin algebra. If    $E_i$ is projective for all $i \geq 0$, then  $\Lambda$ is self-injective.}
\\
\par
 Let $\Lambda$  be    an Artin $R$-algebra over a commutative  Artin ring $R$ with radical $J$.
 There is a  contravariant exact functor $D$: Mod-$\Lambda$  $\rightarrow$ Mod-$\Lambda^{{\rm op}}$
  which is given by  $D$ = Hom$_{R}$($-$, $\mathbb{E}$) where $\mathbb{E}$ is the injective envelope of $R/J$. Note that the functor $D$ induces a classical duality   between
  mod-$\Lambda$    and mod-$\Lambda^{{\rm op}}$.
 Tachikawa \cite{tac} gave the following conjecture which is equivalent to (NC).
 \\
\\
{\bf Tachikawa Conjecture (TC) }\ \ {\it  Let   $\Lambda$  be    an Artin algebra. }
\begin{itemize}
\item[$({\rm TC1})$]If   {\rm Ext}$_{\Lambda}^{i}(D(\Lambda^{{\rm op}}),  \Lambda_{\Lambda}) =  0$  for all  $i \geq 1$, then $\Lambda$ is self-injective.
\item[$({\rm TC2})$] Assume    $\Lambda$ is self-injective. If $M \in$  mod-$\Lambda$   and  {\rm Ext}$_{\Lambda}^{i}(M, M) =  0$  for all  $i \geq 1$, then $M$ is projective.\end{itemize}

In 1975, Auslander and Reiten \cite{ar0}  posed   another conjecture, of which (NC)
is a special case.
 \\
\\
{\bf Generalized Nakayama Conjecture (GNC)  }\ \ {\it  Let   $\Lambda$  be    an Artin algebra.
 If $S \in$  {\rm mod}-$\Lambda$  is simple then there exists some $i \geq 0$  such that  {\rm Ext}$_{\Lambda}^{i}(S, \Lambda)\neq 0$. (This is equivalent to the statement that,
 each indecomposable injective $\Lambda$-module occurs as   a direct summand of some $E_{i}$.)}
\\
\par
It is shown by Auslander and Reiten \cite{ar0} that  (GNC)  is equivalent to
the following conjecture.
 \\
\\
{\bf Auslander-Reiten  Conjecture (ARC) }\ \ {\it  Let   $\Lambda$  be    an Artin algebra.
 If $M \in$  {\rm mod}-$\Lambda$   and  {\rm Ext}$_{\Lambda}^{i}(M, M\oplus\Lambda) =  0$   for all  $i \geq 1$, then $M$ is projective.}
\\
\par
In 1990, Colby and Fuller \cite{cof}  posed   the following conjecture, of which (GNC)
is a special case.
\\
\\
{\bf Strong Nakayama Conjecture (SNC)  }\ \ {\it  Let   $\Lambda$  be    an Artin algebra.
 If $M \in$  {\rm mod}-$\Lambda$    and  {\rm Ext}$_{\Lambda}^{i}(M, \Lambda) =  0$  for all  $i \geq 0$, then $M$ is zero.}
\\
\par
Recall that an  Artin algebra is said to be  $Gorenstein$ provided that both {\rm id}$_{\Lambda}$($\Lambda$)
 and {\rm id}$_{\Lambda^{{\rm op}}}$($\Lambda^{{\rm op}}$) are finite. The Nakayama Conjecture is also   a special case  of the following conjecture posed by  Auslander and Reiten \cite{ar2} in 1994.
\\
\\
{\bf Auslander-Gorenstein Conjecture (AGC)  }\ \ {\it  Let   $\Lambda$  be    an Artin algebra. If    {\rm pd}$_{\Lambda}$$(E_{i}) \leq i$ for all $i \geq 0$, then  $\Lambda$ is  Gorenstein.}
\\
\par
The following   conjecture was posed by  Auslander and Reiten in \cite[p. 150]{ar1} (see also \cite[Conjecture (13)]{ars}).
\\
\\
{\bf Gorenstein Symmetry Conjecture (GSC) }\ \ {\it  Let   $\Lambda$  be    an Artin algebra   with {\rm id}$_{\Lambda}$$(\Lambda) < \infty$. Then {\rm id}$_{\Lambda^{{\rm op}}}$$(\Lambda^{{\rm op}}) < \infty$}.
\\
\par
For an Artin algebra $\Lambda$,   we learn from \cite[p. 121]{ar1}  and  \cite[Corollary 2]{hua} that {\rm id}$_{\Lambda}$$(\Lambda) \leq 1$ if and only if {\rm id}$_{\Lambda^{{\rm op}}}$$(\Lambda^{{\rm op}}) \leq 1$.
Later, Beligiannis \cite[Theorem 11.4]{bel}  proved that (GSC) holds for the class of  virtually Gorenstein Artin  algebras.

It is well known that all the conjectures stated above are formal consequences of the conjecture formulated below.
\\
\\
{\bf Finitistic Dimension Conjecture (FDC) }\ \ {\it  Let   $\Lambda$  be    an Artin algebra. Then} findim$\Lambda$ := sup\{pd$_{\Lambda}(M)$ \ $|$ \ $M \in$  mod-$\Lambda$ with pd$_{\Lambda}(M)    < \infty$\} $< \infty$.
\\
\par
We summarize the relation of each conjecture in the following implication for Artin algebras,
  \[
\begin{array}{ccccccccccc}
\mathrm{(FDC)} & {\xRightarrow{\hspace{2em}}} & \mathrm{(SNC)} &{\xRightarrow{\hspace{2em}}} & \mathrm{(GNC)} &{\xRightarrow{\hspace{2.7em}}} & \mathrm{(AGC)} & \multicolumn{3}{c}{\xRightarrow{\hspace{2.67em}}} & \mathrm{(NC)} \\
\big\Downarrow & & & & \mathrel{\rotatebox[origin=c]{90}{$\Lleftarrow\mspace{-6mu}\Rrightarrow$}} & & & & \mathrel{\rotatebox[origin=c]{90}{$\Lleftarrow\mspace{-6mu}\Rrightarrow$}} \\
\mathrm{(GSC)} & & & & \mathrm{(ARC)} & \multicolumn{3}{c}{\xRightarrow{\hspace{8em}}} & \mathrm{(TC1)}\ \mathrm{and}\ \mathrm{(TC2)}
\end{array}
\]
 where the notation (p) $\Lleftarrow\mspace{-6mu}\Rrightarrow $ (q) means that all  Artin algebras satisfy (p) if and only if all  Artin algebras satisfy (q), while (p) $\Rightarrow$ (q)  means that   every Artin algebra that  satisfies (p) also satisfies (q).

 Auslander, Reiten and Smal${\o}$ \cite{ars} pointed out  the implications  ${\rm (FDC)} \Rightarrow {\rm(GSC)}$ and ${\rm (FDC)} \Rightarrow {\rm(SNC)}$.
The implications  ${\rm (SNC)} \Rightarrow {\rm(GNC)} \Rightarrow {\rm(NC)}$  and  ${\rm (ARC)} \Rightarrow {\rm(TC)}$ are trivial.
Auslander and Reiten established the implication ${\rm (GNC)} \Rightarrow {\rm(AGC)}$ in  \cite[Corollary 5.5(b)]{ar2} and the equivalence
${\rm (GNC)}   \Lleftarrow\mspace{-6mu}\Rrightarrow {\rm (ARC)} $ in  \cite{ar0}. Finally, Tachikawa proved the   equivalence
${\rm (NC)}   \Lleftarrow\mspace{-6mu}\Rrightarrow {\rm (TC)} $ in  \cite[pp. 115-116]{tac}.

Although numerous known cases have been established where one of these conjectures holds (see, for instance, [3-7, 10-17, 24, 27, 29-33,   35, 36] and the references therein),
the full resolution of all these conjectures still remains open.

As an immediate consequence of Theorem \ref{thm81}, we obtain the  validity of  (SNC)  for virtually Gorenstein Artin  algebras (see Corollary \ref{cor10}).

\begin{cor}\label{cor881}   Let  $\Lambda$ be a virtually Gorenstein Artin  algebra,  and    $M\in $  {\rm Mod}-$\Lambda$. If ${\rm Ext}_{\Lambda}^{n}(M, \Lambda)=0$  for
 all $n \geq 0$, then $M  $ is zero.\end{cor}

It follows  that   (GNC), (AGC) and (NC)   are true for      virtually Gorenstein Artin  algebras.
We also  prove  the  validity  of (ARC)  for virtually Gorenstein Artin algebras  of  finite CM-type (see Theorem  \ref{thm501}).
Finally, we provide another proof of (GSC) for virtually Gorenstein Artin algebras (see  Theorem \ref{thm102}),  and obtain the validity of (TC1) for virtually Gorenstein Artin  algebras.

We continue the introduction by recalling some basic definitions and notions
that we use in this paper. Given a $\Lambda$-module $M$ and a class $ \mathscr{C}$ of    $\Lambda$-modules, we denote

  \begin{itemize}
 \item  $^\perp {\mathscr{C}}=\{X: {\rm Ext_{\Lambda}^{1}}(X, C)=0$  for all $C\in
{\mathscr {C}}\}$,
 \item  $ {\mathscr {C}}^\perp=\{X: {\rm Ext_{\Lambda}^{1}}(C, X)=0$
for all $C\in {\mathscr {C}}\}$,
\item  $^{\perp_{\infty}} {\mathscr{C}}=\{X: {\rm Ext}_{\Lambda}^{i}(X, C)=0$  for all $C\in
{\mathscr {C}} $  and $i \geq 1$\},
 \item  $ {\mathscr {C}}^{\perp_{\infty}}=\{X: {\rm Ext}_{\Lambda}^{i}(C, X)=0 $
for all $C\in {\mathscr {C}}$  and $i \geq 1$\},
\item  add$(\mathscr {C}) =$   the class of all direct summands of
finite direct sums of modules in $ \mathscr{C}$,
\item${\bf Proj}(\Lambda)$
= the class of all projective right $\Lambda$-modules,
\item${\bf Inj}(\Lambda)$
= the class of all injective right $\Lambda$-modules.\end{itemize}

{\bf   (Pre)covers and (pre)envelopes.} \ \ Let $\mathscr{C}$    be a  class of modules in  Mod-$\Lambda$.
A  morphism $\phi: C\rightarrow M$ of Mod-$\Lambda$ with $C\in \mathscr{C}$ is called a  $\mathscr{C}$-$precover$ \cite{eno5} of $M$ if for any
morphism $f: C'\rightarrow M$ with $C'\in \mathscr{C}$,
there is a morphism $g : C'\rightarrow C$ such that $\phi g = f$. Moreover, if the only such $g$ are automorphisms of $C$ when
$C'= C$ and $f= \phi $, then the $\mathscr{C}$-precover $\phi$ is
called a $ \mathscr{C}$-$cover$. Dually, we have the
definitions of a  $\mathscr{C}$-preenvelope and a $\mathscr{C}$-envelope.

A  module $M$ is said to have a  {\it special}  $\mathscr{C}$-$precover$, if there is an  exact sequence  $ 0
\rightarrow K\rightarrow C\rightarrow M\rightarrow 0$  with $C\in
\mathscr{C}$ and $K\in
{\mathscr {C}}^\perp$.  $M$ is said to have a  {\it special}  $\mathscr{C}$-$preenvelope$, if there is an  exact sequence  $ 0
\rightarrow M\rightarrow C\rightarrow L\rightarrow 0$  with $C\in
\mathscr{C}$ and $L\in
{^\perp {\mathscr{C}}}$.
\\
\par
{\bf   Cotorsion  pair.}\ \ A pair $(\mathscr{X},\mathscr{Y})$  of subcategories of  Mod-$\Lambda$   is said to be  a  {\it cotorsion  pair} \cite{eno5} if ${\mathscr{X}}^\perp = {\mathscr{Y}}
 $ and $^\perp{\mathscr {Y}}$ = ${\mathscr {X}}$. A cotorsion pair (${\mathscr{X},\ \mathscr{Y}}$)
 is called $complete$   if every object in  Mod-$\Lambda$   has a   special  ${\mathscr{X}}$-precover
and  a   special  $\mathscr
{Y}$-preenvelope.  A
cotorsion pair (${\mathscr{X},\ \mathscr{Y}}$)
 is said to be  $ hereditary$   if whenever
 $0\rightarrow Y\rightarrow Y'\rightarrow
Y''\rightarrow 0$ is exact with $Y$, $Y'\in {\mathscr{Y}}$
then $Y''\in {\mathscr{Y}}$.

\medskip

\section{A characterization of    modules of  virtually finite  injective
  dimension}

The modules in  $^\perp\mathscr{GI}$ are said to be  {\it  of  virtually finite  injective
  dimension} \cite{bel}.
In this section, we   give a characterization of modules of  virtually finite  injective
  dimension over  right Noetherian rings.

We need   the following    proposition  due to $\check{\rm S}$aroch and $\check{\rm S}$$\check{\rm t}$ov$\acute{\rm {\i}}$$\check{\rm c}$ek.
\begin{prop}\label{prop1} {\rm (\cite[Theorem 5.6]{ss})}. For any ring $\Lambda$, $(^\perp\mathscr{GI}$, $\mathscr{GI})$ is a hereditary complete
  cotorsion pair.   \end{prop}

We recall some definitions from \cite{eno5} and \cite{gbe}.

 \begin{defn} \label{defn2} \rm{Let $\mu$ be an ordinal and $\mathscr{A}$ = ($A_{\alpha} $ $|$ $\alpha \leq \mu$) be a sequence of modules.
  If $A_{0} $ = 0,  $A_{\alpha} \subseteq  A_{\alpha + 1}$ for all $\alpha < \mu$ and   $A_{\alpha} $ = $\bigcup_{\beta < \alpha} A_{\beta} $ for all
 limit ordinals $\alpha \leq \mu$, then the sequence $\mathscr{A}$  is called a {\it continuous {\rm (}well ordered{\rm )} chain of modules.}

   Let $M$ be a module and $\mathscr{X}$ be a class of modules. $M$ is  $\mathscr{X}$-{\it filtered}, provided that there is an ordinal $\mu$
   and  a continuous chain of modules, ($M_{\alpha} $ $|$ $\alpha \leq \mu$), consisting of submodules of $M$ such that $M$ = $M_{\mu} $,
    and each of the modules $M_{\alpha + 1}/M_{\alpha}  $  ($\alpha < \mu$) is isomorphic to an element of $\mathscr{X}$. Then chain ($M_{\alpha} $ $|$ $\alpha \leq \mu$)
   is called an $\mathscr{X}$-{\it filtration} of $M$. If $\mu$ is finite, then  $M$ is said to be {\it finitely} $\mathscr{X}$-{\it filtered}.}\end{defn}

\begin{nota}\label{nota11}  {\rm   Let   $\mathscr{X}$ be a class of modules. The class of all finitely  $\mathscr{X}$-filtered modules is denoted by
 filt($\mathscr{X}$)}.
\end{nota}

For convenience, we list the following two results.

\begin{lem}\label{lem2}  {\rm (\cite[Corollary 6.14]{gbe})}. Let   $\mathscr{X}$ be a set of  $\Lambda$-modules containing $\Lambda$.
Then the class  $^\perp(\mathscr{X}^\perp)$ consists of all direct summands of $\mathscr{X}$-filtered modules.
\end{lem}

\begin{lem}\label{lem21}  {\rm (\cite[Corollary 1.14(ii)]{poj})}. Let   $\mathscr{X}$ be a set of finitely presented  $\Lambda$-modules containing $\Lambda$.
Then the class   {\rm add(filt($\mathscr{X}$))} coincides with the class of all finitely presented modules in $^\perp(\mathscr{X}^\perp)$.
\end{lem}

The following homological lemma is a generalization of \cite[Theorem 7.3.4]{eno5}.

\begin{lem}\label{lem1}  {\rm (\cite[Lemma 3.8]{mao}).} Let $M$ and $N$ be   $\Lambda$-modules, let   $n$ be a positive integer,   and suppose that $M$ is
the union of a continuous chain of submodules $(M_{\alpha} $ $|$ $\alpha \leq \kappa)$. If ${\rm Ext}_{\Lambda}^{n}(M_{0}, N)=0$ and ${\rm Ext}_{\Lambda}^{n}(M_{\alpha + 1}/M_{\alpha}, N)=0$
whenever  $\alpha  < \kappa$\, then ${\rm Ext}_{\Lambda}^{n}(M, N)=0$. \end{lem}

To simplify the proof  of the next theorem, we give the following notation.
\begin{nota}\label{nota1}  {\rm  \  Let  $\Lambda$ be a right  Noetherian ring. Then there is a
family ($E_{j})_{j\in J}$ of indecomposable injective right  $\Lambda$-modules such that every injective right  $\Lambda$-module is the direct sum
of copies of the various $E_{j}$ (see the proof of \cite[Theorem 5.4.1]{eno5}). For each $E_{j}$, take a  projective resolution $$ \cdots \stackrel {f_{j2}} \longrightarrow
{F_{j1}}\stackrel {f_{j1}} \longrightarrow {F_{j0}}\stackrel {f_{j0}}
\longrightarrow E_{j}  \longrightarrow  0  \eqno (\spadesuit_{j})$$
 of  $E_{j}$, and let
 $$ 0  \longrightarrow \Lambda_{\Lambda}\stackrel {h_{0}}
\longrightarrow
E_0\stackrel {h_{1}} \longrightarrow E_1\stackrel {h_{2}}
\longrightarrow \cdots \eqno (\clubsuit)$$
  be a minimal  injective resolution of   $\Lambda_{\Lambda}$. Put $\mathscr{X}$ = \{im$(f_{ji})$, im$(h_{i})$ \ $|$\  $i\geq 0, j\in J $\}}.\end{nota}

 Recall that a class   $\mathscr{C} \subseteq$    Mod-$\Lambda$    is  {\it  thick} if it is
closed under direct summands and has the two out of three property: for every exact
sequence $0\rightarrow  A   \rightarrow B \rightarrow  C\rightarrow 0   $ in  Mod-$\Lambda$  with two terms in  $\mathscr{C}$ , the third term belongs
to $\mathscr{C}$ as well. Now  we give a characterization of modules of virtually finite injective dimension   over right
Noetherian rings.

\begin{thm}\label{thm1}  Let  $\Lambda$ be a right  Noetherian ring and keep the notation as above. Then
$(^\perp(\mathscr{X}^\perp)$, $\mathscr{X}^\perp)$ is a hereditary complete
  cotorsion pair. Moreover,  $^\perp\mathscr{GI}$ = $^\perp(\mathscr{X}^\perp)$, i.e.,
 $\mathscr{X}^\perp$ = $\mathscr{GI}$. \end{thm}

\begin{proof}
Since $\mathscr{X}$ is a set, $(^\perp(\mathscr{X}^\perp)$, $\mathscr{X}^\perp)$ is a   complete
  cotorsion pair by  \cite[Theorem 7.4.1]{eno5}. Next we prove that it is hereditary.

  Since  each im$(f_{ji})$ belongs to $^\perp(\mathscr{X}^\perp)$, for any $n \geq 1$, we deduce from the exact sequence $(\spadesuit_{j})$
that ${\rm Ext}_{\Lambda}^{n}($im$(f_{ji})$, $Y)=0$ for all $Y \in \mathscr{X}^\perp$. Therefore, im$(f_{ji}) \in$  $^{\perp_{\infty}}(\mathscr{X}^\perp)$.
On the other hand, every injective right  $\Lambda$-module is the direct sum
of copies of the various $E_{j}$. Hence   ${\bf Inj}(\Lambda) \subseteq $  $^{\perp_{\infty}}(\mathscr{X}^\perp)$.
But then we see  from the exact sequence $(\clubsuit)$
that  im$(h_{i}) \in$  $^{\perp_{\infty}}(\mathscr{X}^\perp)$.
   In addition, the class  $^\perp(\mathscr{X}^\perp)$ consists of all direct summands of $\mathscr{X}$-filtered modules by Lemma \ref{lem2}.
These observations  together with Lemma \ref{lem1} tell us that, ${\rm Ext}_{\Lambda}^{n}($$B$, $Y)=0$ for all $B \in {^\perp(\mathscr{X}^\perp)}$, $Y \in \mathscr{X}^\perp$,  and $n \geq 1$.
So the pair $(^\perp(\mathscr{X}^\perp)$, $\mathscr{X}^\perp)$   is hereditary by \cite[Lemma 5.24]{gbe}.

Now we prove that  $\mathscr{X}^\perp$ = $\mathscr{GI}$.
Since injective modules are contained in the thick class $^\perp\mathscr{GI}$ (cf. \cite[Lemma 5.24]{ss}), we see from the exact sequences $(\spadesuit_{j})$  and $(\clubsuit)$
that $\mathscr{X} \subseteq {^\perp\mathscr{GI}}$. Hence $\mathscr{X}^\perp \supseteq {(^\perp\mathscr{GI}})^\perp$ = $\mathscr{GI}$ (cf. Proposition \ref{prop1}).

It remains to show that $\mathscr{X}^\perp $   $\subseteq \mathscr{GI}$. For any $M \in \mathscr{X}^\perp $, there is an exact sequence
$0\rightarrow  K    \rightarrow \Lambda^{(\lambda)} \rightarrow  M\rightarrow 0   $  for some set $\lambda$. Then the sequence
$0\rightarrow  \Lambda^{(\lambda)}  \rightarrow E_{0}^{(\lambda)} \rightarrow {\rm(im(}h_{1}))^{(\lambda)}\rightarrow 0   $  is also exact.
 Constructing the pushout of  $\Lambda^{(\lambda)}  \rightarrow E_{0}^{(\lambda)}$ along $\Lambda^{(\lambda)} \rightarrow  M$  yields the
following commutative diagram with exact
rows  and columns:

 $$\xymatrix{&&0\ar[d]&0\ar@{-->}[d]&\\0\ar[r]& K\ar[r]
\ar@{=}[d]&\Lambda^{(\lambda)}\ar[d]\ar[r]
&M\ar@{-->}[d]\ar[r]&0\\
0\ar@{-->}[r]&  K\ar@{-->}[r]
& E_{0}^{(\lambda)}\ar[d]\ar@{-->}[r]
& Q\ar@{-->}[d]\ar@{-->}[r]&0\\
 & &{\rm im(}h_{1})^{(\lambda)}\ar[d]\ar@{=}[r]&{\rm im(}h_{1})^{(\lambda)}\ar@{-->}[d]&&&\\
&&0&0}$$
Note that ${\rm Ext}_{\Lambda}^{1}($${\rm im(}h_{1})^{(\lambda)}$, $M)   \cong  $ $({\rm Ext}_{\Lambda}^{1}$${\rm(im(}h_{1})$, $M))^{\lambda} =  0$.
Thus  $Q \cong M \oplus {\rm im(}h_{1})^{(\lambda)}$, and we get   an epimorphism $E_{0}^{(\lambda)} \rightarrow  M$ with $E_{0}^{(\lambda)}  $ injective ($\Lambda$ is right  Noetherian).
Hence there exists a short exact sequence   $$0 \longrightarrow K_{0} \longrightarrow I_{0}  \stackrel {g_{0}}\longrightarrow M \longrightarrow 0$$
where  $g_{0}$ is an  injective cover.  Then $K_{0}\in {\bf Inj}(\Lambda)^\perp$ by \cite[Corollary 7.2.3]{eno5}. But   $M \in {\bf Inj}(\Lambda)^{\perp_{\infty}}$.
 Thus, one can see from the above short exact sequence that $K_{0} \in  {\bf Inj}(\Lambda)^{\perp_{\infty}}$. Hence  $ {\rm Ext}_{\Lambda}^{1}{\rm(im}(f_{ji}), K_{0}) \cong {\rm Ext}_{\Lambda}^{i+1}{\rm (}E_{j}, K_{0}) = 0.$

 For any $i \geq 1$, we have an exact sequence
  $$0 = {\rm Ext}_{\Lambda}^{1}{\rm(im}(h_{i}), M)  \rightarrow  {\rm Ext}_{\Lambda}^{2}{\rm(im}(h_{i}), K_{0})  \rightarrow  {\rm Ext}_{\Lambda}^{2}{\rm(im}(h_{i}), I_{0})  = 0.$$
So  $  {\rm Ext}_{\Lambda}^{2}{\rm(im}(h_{i}), K_{0})   = 0$.  Now consider the short exact sequence
$0\rightarrow  {\rm im(}h_{i - 1})  \rightarrow E_{i - 1}\rightarrow {\rm im(}h_{i})\rightarrow 0   $
which  induces an exact sequence  $$0 = {\rm Ext}_{\Lambda}^{1}(E_{i - 1}, K_{0})  \rightarrow  {\rm Ext}_{\Lambda}^{1}{\rm(im}(h_{i-1}), K_{0})  \rightarrow  {\rm Ext}_{\Lambda}^{2}{\rm(im}(h_{i}), K_{0})  = 0.$$
Thus $  {\rm Ext}_{\Lambda}^{1}{\rm(im}(h_{i-1}), K_{0})   = 0$ for any $i \geq 1$. This implies  $K_{0} \in \mathscr{X}^{\perp}$.
Recursively, we construct an exact sequence of right $\Lambda$-modules
$$ \cdots \stackrel {g_{2}} \longrightarrow
I_1\stackrel {g_{1}} \longrightarrow I_0\stackrel {g_{0}}
\longrightarrow M  \longrightarrow  0  $$
with $I_i \in {\bf Inj}(\Lambda)$ and ker$(g_i) \in {\bf Inj}(\Lambda)^{\perp_{\infty}}$ ($i \geq 0$).
Therefore, $M $   $\in \mathscr{GI}$ by \cite[Lemma 17(2)]{gii}, and   $\mathscr{X}^\perp $   $\subseteq \mathscr{GI}$, as required. \end{proof}

\begin{prop}\label{prop301}  Let  $\Lambda$ be an Artin algebra and keep the notation as above. Then
 $^\perp \mathscr{X} =$   $\mathscr{GP}$. \end{prop}

\begin{proof} Since $\Lambda$ is an Artin algebra, the set $J$  in Notation \ref{nota1} is finite.
In addition,   every  module over an Artin algebra admits a  projective envelope by \cite[Proposition 3.5]{am}, and  ${\bf Proj}(\Lambda)$
is  closed under direct products.
 Then the argument similar to that of the previous
theorem shows that $^\perp \mathscr{X} =$   $\mathscr{GP}$.
\end{proof}
\medskip

\section{ Virtually Gorenstein Artin algebras are weakly Gorenstein}

We begin with the following   remark.

\begin{rem}\label{rem21}  {\rm Let $\mathscr{X}$ be the set as in  Notation \ref{nota1}.  If    $\Lambda$ is an Artin algebra,
then we can always assume that $\mathscr{X} \subseteq $ mod-$\Lambda$.}\end{rem}

In view of the above remark, we get the following lemma.
\begin{lem}\label{lem22}   Let    $\Lambda$ be a virtually
Gorenstein Artin algebra and $\mathscr{X}$ be the set as in {\rm Notation \ref{nota1}}. Then, for any  $M\in $  {\rm mod}-$\Lambda$, there are two short exact sequences in  {\rm mod}-$\Lambda$
 $$0 \rightarrow M  \rightarrow Q  \rightarrow B \rightarrow 0 \ \ \ \ \ \ and  \ \ \ \ \ \ 0 \rightarrow M  \rightarrow G  \rightarrow C \rightarrow 0$$
 where  $Q  $ and  $C  $ belong to {\rm add(filt($\mathscr{X}$))},  $B\in $  $\mathscr{GP}$ and $G\in $  $\mathscr{GI}$. \end{lem}

 \begin{proof} Since $\Lambda$ is  virtually
Gorenstein, we have a  hereditary  cotorsion triple ($\mathscr{GP}$,  $^\perp\mathscr{GI}$, $\mathscr{GI}$).
 Hence, from \cite[Theorem 8.2(vi)]{bel}, we get two short exact sequences in  {\rm mod}-$\Lambda$
 $$0 \rightarrow M  \rightarrow Q  \rightarrow B \rightarrow 0 \ \ \ \ \ \ and  \ \ \ \ \ \ 0 \rightarrow M  \rightarrow G  \rightarrow C \rightarrow 0$$
 with  $Q  $ and  $C  $ belong to $^\perp\mathscr{GI}$,  $B\in $  $\mathscr{GP}$ and $G\in $  $\mathscr{GI}$.
Note that  $^\perp\mathscr{GI}$ = $^\perp(\mathscr{X}^\perp)$ by Theorem \ref{thm1}. It follows from   Lemma  \ref{lem21} that  $Q  $ and  $C  $ are contained in  {\rm add(filt($\mathscr{X}$))},
finishing the proof.\end{proof}

To determine the  projective and injective dimensions for   modules over a  virtually
Gorenstein Artin algebra,  we need three auxiliary lemmas.

\begin{lem}\label{lem81}   Let    $\Lambda$ be a virtually
Gorenstein Artin algebra and   $M\in $  {\rm Mod}-$\Lambda$. If   $M \in $ ${D(\Lambda^{{\rm op}})}^{\perp_{\infty}}\ \cap $ $\mathscr{GP}^\perp$,
  then $M $ is injective.\end{lem}

 \begin{proof}
Since  $\Lambda$ is an Artin algebra, there is a
family ($E_{j})_{j\in J}$ of finitely generated injective right  $\Lambda$-modules such that every injective right  $\Lambda$-module is the direct sum
of copies of the various $E_{j}$.
    Note that  $D(\Lambda^{{\rm op}})  $ is
an injective cogenerator and  $M \in $ ${D(\Lambda^{{\rm op}})}^{\perp_{\infty}} $.  So $M \in $  ${E_{j}}^{\perp_{\infty}}$  for all  $j \in J$. It follows that     $M \in $  {\bf Inj}$(\Lambda)^{\perp_{\infty}}$. We then deduce from
the  exact sequences $(\spadesuit_{j})$ and  $(\clubsuit)$
 in Notation \ref{nota1}  that $ {\rm Ext}_{\Lambda}^{l}{\rm(im}(f_{ji}), M) \cong {\rm Ext}_{\Lambda}^{l+i}{\rm (}E_{j}, M) = 0 $ ($j\in J$)
 and $ {\rm Ext}_{\Lambda}^{l+1}{\rm(im}(h_{1}), M) \cong {\rm Ext}_{\Lambda}^{l+1}{\rm (}E_{0}, M) = 0 $ for all $l \geq  1$, respectively.
 Then  $ {\rm Ext}_{\Lambda}^{l+m}{\rm(im}(h_{m}), M) \cong {\rm Ext}_{\Lambda}^{l+1}{\rm(im}(h_{1}), M) = 0 $ for all $l \geq  1$ and every $m \geq  2$.

 Let $\mathscr{S}$ be an irredundant finite set   of representatives of the simple modules in Mod-$\Lambda$.  For any    $S\in $  $\mathscr{S}$,  there is a short exact sequence
$0\rightarrow  S    \rightarrow Q \rightarrow  B\rightarrow 0   $ with  $Q\in $    add(filt($\mathscr{X}$))   and   $B\in $  $\mathscr{GP}$ by  Lemma  \ref{lem22}.
Then $Q$ is  isomorphic to a  direct summand of some module $Y$  in
 filt($\mathscr{X}$). There exists a submodule chain   $$0 = Y_{0} \subseteq Y_{1} \subseteq \cdots \subseteq Y_{n-1} \subseteq  Y_{n} = Y$$
 of  finite length such that
  each of the modules $L_{ i }:= Y_{ i }/Y_{i-1}  $  ($1\leq i \leq n$) is isomorphic to an element of $\mathscr{X}$.
In the previous paragraph we have shown that, for every $L_{ i }$, there exists some positive integer $t_{i}$ such that   {\rm Ext}$ _{\Lambda}^{t_{i}+l}(L_{ i }, M)=0$  for all  $l \geq 1$.
So we can immediately deduce  from the short exact sequence
$0\rightarrow L_{1} = Y_{ 1 }    \rightarrow Y_{ 2 } \rightarrow  L_{2}\rightarrow 0   $ that
 {\rm Ext}$ _{\Lambda}^{t_{S}+l}(Y_{2}, M)=0$  for all  $l \geq 1$, where $t_{S}$ = sup\{$t_{i}$ \ $|$\   $1\leq i \leq n$ \} $  < \infty$.
Inductively, we get  that {\rm Ext}$ _{\Lambda}^{t_{S}+l}(Y, M)=0$  for all  $l \geq 1$.
It follows that  {\rm Ext}$ _{\Lambda}^{t_{S}+l}(Q, M)=0$  for all  $l \geq 1$.
In addition,  $M \in $   $\mathscr{GP}^\perp = \mathscr{GP}^{\perp_{\infty}}$  (cf. \cite[Theorem 3.5(i)]{bel} and   $B \in \mathscr{GP}$.  So
   $M \in B^{\perp_{\infty}}$. Hence  we see  from the short exact sequence $0\rightarrow  S    \rightarrow Q \rightarrow  B\rightarrow 0   $
that {\rm Ext}$ _{\Lambda}^{t_{S}+l}(S, M)=0$  for all  $l \geq 1$.

Since the set   $\mathscr{S}$ is finite,   $t$ = sup\{$t_{S} + 1$ \ $|$\   $S\in $  $\mathscr{S}$ \} $  < \infty$. In addition,  any $V \in$ {\rm mod}-$\Lambda$
 is  finitely  $\mathscr{S}$-filtered by \cite[Proposition
11.1]{and}. Thus, as above, one can easily obtain that  {\rm Ext}$ _{\Lambda}^{t}(V, M)=0$  for any
$V \in$ {\rm mod}-$\Lambda$. Therefore, {\rm id}$_{\Lambda}$($M$) $\leq t <\infty$, and we get an exact sequence
 $$0  \longrightarrow{M}\stackrel {g_{0}} \longrightarrow{I_{0}}\stackrel {g_{1}} \longrightarrow \cdots \stackrel {g_{t}} \longrightarrow
{I_{t}} \longrightarrow  0  $$ with   $I_{i} $ injective ($0\leq i \leq t$).
But we have shown that   $M \in $  {\bf Inj}$(\Lambda)^{\perp_{\infty}}$. Therefore, applying the functor $ {\rm Hom}_{\Lambda}(-, M)$ to the
exact sequence
 $$0  \longrightarrow{{\rm im}(g_{1})}\stackrel {g_{1}} \longrightarrow{I_{1}}\stackrel {g_{2}} \longrightarrow \cdots \stackrel {g_{t}} \longrightarrow
{I_{t}} \longrightarrow  0,  $$
we obtain that
$ {\rm Ext}_{\Lambda}^{1}{\rm(im}(g_{1}), M) \cong {\rm Ext}_{\Lambda}^{t}{\rm (}I_{t}, M) = 0.$  This implies that $M $ is  a direct summand of the
injective module $I_{0} $,  finishing  the proof.
 \end{proof}

\begin{lem}\label{lem82}   Let    $\Lambda$ be a virtually
Gorenstein Artin algebra and   $M\in $  {\rm Mod}-$\Lambda$. If   $M \in $ ${D(\Lambda^{{\rm op}})}^{\perp_{\infty}} $,
  then $M $   $\in $  $\mathscr{GI}$.\end{lem}

\begin{proof} By Proposition \ref{prop1}, there is a short exact sequence
$0\rightarrow  K    \rightarrow I \rightarrow  M\rightarrow 0   $ with  $I \in $    $^\perp\mathscr{GI}$ = $^{\perp_{\infty}}{\mathscr{GI}}$   and   $ K \in $  $\mathscr{GI} \subseteq$ ${\bf Inj}(\Lambda)^{\perp_{\infty}}$.
Then $I \in $ ${D(\Lambda^{{\rm op}})}^{\perp_{\infty}} $ since $M \in $ ${D(\Lambda^{{\rm op}})}^{\perp_{\infty}} $.
But $^{\perp}{\mathscr{GI}}$ = $\mathscr{GP}^\perp$ as $\Lambda$ is virtually
Gorenstein. Hence  $I \in $ ${D(\Lambda^{{\rm op}})}^{\perp_{\infty}}\ \cap $ $\mathscr{GP}^\perp$.
It follows from  Lemma  \ref{lem81} that $I  $ is injective.
Therefore, $M $   $\in \mathscr{GI}$  since the cotorsion pair $(^\perp\mathscr{GI}$, $\mathscr{GI})$ is   hereditary by Proposition  \ref{prop1}, as required. \end{proof}

\begin{lem}\label{lem83}   Let    $\Lambda$ be a virtually
Gorenstein Artin algebra and   $M\in $  {\rm Mod}-$\Lambda$. If   $M \in $ ${\mathscr{GI}}^{\perp_{\infty}} $,
  then $M $   is injective.\end{lem}

\begin{proof} Since $M \in $ ${\mathscr{GI}}^{\perp_{\infty}} \subseteq$ ${D(\Lambda^{{\rm op}})}^{\perp_{\infty}} $,
  $M $   $\in \mathscr{GI}$ by Lemma \ref{lem82}.
So  $E(M)/M  $ $\in \mathscr{GI}$, where  $E(M)$ is the injective envelope of $M $.
 It follows that $M $ is a direct summand of $E(M)$ since  Ext$_{\Lambda}^{1}(E(M)/M, M) =  0$,  as desired. \end{proof}

Let $  \mathbb{N}$ be the set  of all non-negative integers and write inf $\emptyset := \infty$.  We now establish a new criterion for determining the (Gorenstein) injective dimensions of   modules over virtually
Gorenstein Artin algebras.

 \begin{thm}\label{thm91}   Let    $\Lambda$ be a virtually
Gorenstein Artin algebra.  For any   $M\in $  {\rm Mod}-$\Lambda$, the following  two equalities hold.
 \begin{itemize}
\item[$(1)$]  {\rm  id$_{\Lambda}$$(M)$ =  inf\{$n \in \mathbb{N}$\ $|$\   Ext$_{\Lambda}^{n+l}(G, M) = 0$ for every $G $   $\in $  $\mathscr{GI}$ and all $l\geq 1$\}.}
\item[$(2)$] {\rm  Gid$_{\Lambda}$$(M)$ =  inf\{$n \in \mathbb{N}$\ $|$\   Ext$_{\Lambda}^{n+l}(D(\Lambda^{{\rm op}}), M) =  0$ for all $l\geq 1$\}.}
\end{itemize}    \end{thm}

\begin{proof} (1). Put $t$ := inf\{$n \in \mathbb{N}$\ $|$\   Ext$_{\Lambda}^{n+l}(G, M) = 0$ for every $G $   $\in $  $\mathscr{GI}$ and all $l\geq 1$\}.
It is clear that $t$ $\leq  $    id$_{\Lambda}$$(M)$. Next   we prove that   id$_{\Lambda}$$(M)$ $\leq t$. We may assume that $t$ $<  \infty$.
Take an injective resolution $$ 0  \longrightarrow M\stackrel {g_{0}}
\longrightarrow
I_0\stackrel {g_{1}} \longrightarrow I_1\stackrel {g_{2}}
\longrightarrow \cdots $$
of   $M$, and set $L_{t}$ := im$(g_{t})$.
 Then $ {\rm Ext}_{\Lambda}^{l}(G, L_{t}) \cong {\rm Ext}_{\Lambda}^{t+l}{\rm (}G, M) = 0$ for every $G $   $\in $  $\mathscr{GI}$ and all $l\geq 1$.
 So  $L_{t}$ is injective by Lemma \ref{lem83}, and   id$_{\Lambda}$$(M) \leq t$. This proves (1).

   The proof of (2) is similar
(but using Lemma \ref{lem82}).
\end{proof}

We next establish the dual of Theorem \ref{thm91}.
 \begin{thm}\label{thm92}   Let    $\Lambda$ be a virtually
Gorenstein Artin algebra.  For any   $M\in $  {\rm Mod}-$\Lambda$, the following  two equalities hold.
 \begin{itemize}
\item[$(1)$]  {\rm  pd$_{\Lambda}$$(M)$ =  inf\{$n \in \mathbb{N}$\ $|$\   Ext$_{\Lambda}^{n+l}(M, Q) = 0$ for every $Q $   $\in $  $\mathscr{GP}$ and all $l\geq 1$\}.}
\item[$(2)$] {\rm  Gpd$_{\Lambda}$$(M)$ =  inf\{$n \in \mathbb{N}$\ $|$\   Ext$_{\Lambda}^{n+l}(M, \Lambda) =  0$ for all $l\geq 1$\}.}
\end{itemize}    \end{thm}

\begin{proof}   (1).   Write  $t$ := inf\{$n \in \mathbb{N}$\ $|$\   Ext$_{\Lambda}^{n+l}(M, Q) = 0$ for every $Q $   $\in $  $\mathscr{GP}$ and all $l\geq 1$\}.
It is   clear that $t \leq$   pd$_{\Lambda}$$(M)$. Next we prove that    pd$_{\Lambda}$$(M)$ $\leq t $. We may
assume $t < \infty$. Let  $G $ be   an arbitrary Gorenstein injective right $\Lambda^{{\rm op}}$-module.
Then $D(G)  \in \mathscr {GP}$  by \cite[Lemma 8.6(1)]{bel}. Thus
$${\rm Ext}_{\Lambda^{{\rm op}}}^{t+l}(G, D(M)) \cong  D{\rm Tor}_{t+l}^{\Lambda}(M, G)  \cong      {\rm Ext}_{\Lambda}^{t+l}(M, D(G)) =  0$$ for all $l \geq 1 $
by  \cite[ Lemma 2.16(b)]{gbe}, and hence     id$_{\Lambda^{{\rm op}}}$$(D(M)) \leq t$   by the left-hand counterpart of Theorem \ref{thm91}(1).
So  the  flat  dimension of $DD(M)$ is   less than or equal to $t$ by \cite[Proposition 3.3(3)]{li}. In addition, $M  $ is a pure submodule of $DD(M)$ by \cite[Proposition 5.3.9]{eno5}. Hence the  flat  dimension of $M$ is also   less than or equal to $t$ by \cite[Lemma 3.6(2)]{li}.
As is well known, flat modules over an Artin algebra are   projective,
we have that  pd$_{\Lambda}$$(M)$ $\leq t $. This implies that  pd$_{\Lambda}$$(M)$ $= t $.

 (2).   Put  $m$ :=  inf\{$n \in \mathbb{N}$\ $|$\   Ext$_{\Lambda}^{n+l}(M, \Lambda) =  0$ for all $l\geq 1$\}. Obviously,   $m \leq$  Gpd$_{\Lambda}$$(M)$.
 Now we show that   Gpd$_{\Lambda}$$(M)$ $\leq m $. We may assume $m < \infty$. Take a  projective resolution $$\cdots\stackrel {f_{m+2}} \longrightarrow{P_{m+1}}\stackrel {f_{m+1}} \longrightarrow{P_{m}}\stackrel {f_{m}} \longrightarrow \cdots \stackrel {f_{2}} \longrightarrow
{P_{1}}\stackrel {f_{1}} \longrightarrow {P_{0}}\stackrel {f_{0}}
\longrightarrow M  \longrightarrow  0  $$
 of  $M$, and write $K_{m}$ := im$(f_{m})$. We have to prove that  $K_{m}$ is Gorenstein projective.
  Note that $\Lambda \cong $ $DD(\Lambda)$. So, by     \cite[ Lemma 2.16(b)]{gbe}
we have $${\rm Ext}_{\Lambda^{{\rm op}}}^{l}(D(\Lambda), D(K_{m})) \cong  D{\rm Tor}_{l}^{\Lambda}(K_{m}, D(\Lambda))  \cong
  {\rm Ext}_{\Lambda}^{l}(K_{m}, DD(\Lambda)) \cong    {\rm Ext}_{\Lambda}^{l}(K_{m}, \Lambda)  \cong    {\rm Ext}_{\Lambda}^{m+l}(M, \Lambda) = 0$$ for all $l \geq 1 $.
  In addition,   $\Lambda^{{\rm op}}$ is also virtually Gorenstein (see \cite[Theorem 8.7]{bel}). Thus we infer from   the left-hand counterpart of
  Lemma \ref{lem82}  that  $D(K_{m})$   is Gorenstein injective. Hence  $DD(K_{m})$ is Gorenstein projective by \cite[Lemma 8.6(1)]{bel}.
 But  $K_{m}$ is a pure submodule of $DD(K_{m})$ by \cite[Proposition 5.3.9]{eno5}. Therefore,
$K_{m}$ is Gorenstein projective by \cite[Proposition 3.8(i)]{bel},  i.e.,  Gpd$_{\Lambda}$$(M)$ $\leq m $. This finishes the proof.\end{proof}

 We are in a position to show that every virtually Gorenstein  Artin algebra is weakly Gorenstein.

 \begin{thm}\label{thm101} Let   $\Lambda$  be  a virtually Gorenstein  Artin algebra. Then  $^{\perp_{\infty}}{\Lambda} $ $=$  $\mathscr{GP}$ and ${D(\Lambda^{{\rm op}})}^{\perp_{\infty}} $ $=$  $\mathscr{GI}$.
 In particular,  $\Lambda$ is weakly Gorenstein.
 \end{thm}

 \begin{proof} Since $\Lambda$ is a virtually Gorenstein  Artin algebra, we obtain from Theorems \ref{thm92}(2) and \ref{thm91}(2)   that $^{\perp_{\infty}}{\Lambda} $ $=$  $\mathscr{GP}$
 and ${D(\Lambda^{{\rm op}})}^{\perp_{\infty}} $ $=$  $\mathscr{GI}$, respectively.
 So $\Lambda$ is right weakly Gorenstein. But $\Lambda^{{\rm op}}$ is also   virtually Gorenstein (cf. \cite[Theorem 8.7]{bel}).
 Thus $\Lambda^{{\rm op}}$ is also right weakly Gorenstein. Therefore,  $\Lambda$ is weakly Gorenstein.\end{proof}

As shown in the next example, weakly Gorenstein Artin algebras need not be virtually Gorenstein.
\begin{exm} \label{exm1} According to  {\rm \cite[Theorem 1.4]{kim}}, there is a weakly Gorenstein Artin algebra $\Lambda$ such that
$\mathscr{GP}$ $\cap$  {\rm  mod-}$\Lambda$ $\subseteq {\bf Proj}(\Lambda)$ but $\mathscr{GP}$ $\neq$ ${\bf Proj}(\Lambda)$.
Then $\Lambda$ is not virtually Gorenstein. \end{exm}
 \begin{proof} Assume to the contrary that $\Lambda$ is   virtually Gorenstein. Then any Gorenstein projective module $M$ is a direct limit
 of modules from $\mathscr{GP}$ $\cap$  {\rm  mod-}$\Lambda$ $\subseteq {\bf Proj}(\Lambda)$ by \cite[Theorem 5]{bek}. This implies that $M$ is projective, contradicting   the assumption  that  $\mathscr{GP}$ $\neq$ ${\bf Proj}(\Lambda)$.
 Therefore, $\Lambda$ is not virtually Gorenstein.\end{proof}

The corollary below  is tightly linked to     Question \ref{que1}.
 \begin{cor}\label{cor101} Let   $\Lambda$ be an  Artin algebra with ${\bf Proj}(\Lambda)$ $=$  $\mathscr{GP}$. Then   $\Lambda$ is weakly Gorenstein.
 \end{cor}

  \begin{proof} Note that  ${\bf Proj}(\Lambda)$ $=$  $\mathscr{GP}$. So $\mathscr{GP}^\perp$  = Mod-$\Lambda$, and hence $\mathscr{GP}^\perp $ $\cap$ mod-$\Lambda$  = mod-$\Lambda$
  is {\it contravariantly  finite} (in the sense of \cite{ar1}) in mod-$\Lambda$.
    We then infer from \cite[Theorem 8.2(viii)]{bel} that
  $\Lambda$ is  virtually Gorenstein.    Therefore,  $\Lambda$ is weakly Gorenstein by Theorem \ref{thm101}.\end{proof}

\begin{rem}\label{rem930}  {\rm We note that Corollary \ref{cor101}  does not settle   Question \ref{que1} in   full. That is to say,  Question \ref{que1} remains open.
}\end{rem}

We define an  analogue  of the Strong Nakayama Conjecture   for arbitrary modules as follows: Let    $\Lambda$ be an
  Artin algebra;  for any   $M\in $  {\rm Mod}-$\Lambda$, if    {\rm Ext}$_{\Lambda}^{i}(M, \Lambda) =  0$ for all  $i \geq 0$, then $M$ is zero.

Perhaps surprisingly,  we see from the following result that the analogue the Strong Nakayama Conjecture  for arbitrary modules is valid  over virtually Gorenstein Artin algebras.
\begin{cor}\label{cor10}   Let    $\Lambda$ be a virtually
Gorenstein Artin algebra.  For any   $M\in $  {\rm Mod}-$\Lambda$, if   {\rm Ext}$_{\Lambda}^{i}(M, \Lambda) =  0$  for all  $i \geq 0$, then $M$ is zero.    \end{cor}
\begin{proof}  Since $\Lambda$ is virtually
Gorenstein  and $M\in $  $^{\perp_{\infty}}{\Lambda} $, we obtain from Theorem \ref{thm101} that  $M\in $ $\mathscr{GP}$. But  ${\rm Hom_{\Lambda}}(M, \Lambda) \cong  {\rm Ext} _{\Lambda}^{0}(M, \Lambda)=0$.
Therefore,  $M$ is zero.\end{proof}

\begin{rem}\label{rem29}  {\rm
From the previous corollary and the   implications  ${\rm (SNC)} \Rightarrow {\rm(GNC)} \Rightarrow {\rm(AGC)} \Rightarrow {\rm(NC)}$,  we get    the validities
of (GNC), (AGC) and (NC)   for    virtually
Gorenstein Artin algebras, refining a main result in \cite[Theorem 1.2]{ch1}.}\end{rem}

We define an  analogue  of the Auslander-Reiten Conjecture for arbitrary modules as follows: Let    $\Lambda$ be an
  Artin algebra; for any   $M\in $  {\rm Mod}-$\Lambda$, if  {\rm Ext}$_{\Lambda}^{i}(M, M\oplus\Lambda) =  0$    for all  $i \geq 1$, then $M$ is projective.

We see from the next result that the   analogue the Auslander-Reiten Conjecture for arbitrary modules is true    over an Artin algebra  $\Lambda$  with ${\bf Proj}(\Lambda)$ $=$  $\mathscr{GP}$.
 \begin{cor}\label{cor201} Let   $\Lambda$ be an  Artin algebra with ${\bf Proj}(\Lambda)$ $=$  $\mathscr{GP}$. For any $M \in$  {\rm Mod}-$\Lambda$, we have
       $${\rm pd}_{\Lambda}(M) =  {\rm inf}\{n \in \mathbb{N}\ | \   {\rm Ext}_{\Lambda}^{n+l}(M, \Lambda) =  0 \ {\rm for \ all}\ l\geq 1\}.$$
 \end{cor}

  \begin{proof} Since  ${\bf Proj}(\Lambda)$ $=$  $\mathscr{GP}$,
  $\Lambda$ is  virtually Gorenstein.   Now  one can apply Theorem \ref{thm92} (2) to get that
inf\{$n \in \mathbb{N}$\ $|$\   Ext$_{\Lambda}^{n+l}(M, \Lambda) =  0$ for all $l\geq 1$\}  =  Gpd$_{\Lambda}$$(M)$ =   pd$_{\Lambda}$$(M)$.
  \end{proof}

As a special case of the Auslander-Reiten Conjecture, the following conjecture  was proposed by Luo and Huang  in  \cite{luh}.
 \\
\\
{\bf Gorenstein Projective Conjecture (GPC)  }\ \ {\it  Let   $\Lambda$  be    an Artin algebra.  If   $M $ is a finitely generated Gorenstein projective  $\Lambda$-module
 with   $M \in M^{\perp_{\infty}}$, then $M$ is projective}.
\begin{rem}\label{rem239}  {\rm Using \cite[Theorem 4.6]{luh}, Luo and Huang  proved in \cite[Theorem 4.7]{luh} that   (ARC)  is true for commutative Artinian rings.
However,   as noted  by B$\ddot{\rm o}$hmler and Marczinzik in \cite[page 2, line -16]{bm}, the argument in  \cite[Theorem 4.6]{luh} does not hold in general,
and a concrete counterexample is explicitly constructed   in \cite[Theorem 0.2]{bm}.
}\end{rem}

The implication   (ARC) $\Rightarrow$ (GPC) is   straightforward. Conversely, we obtain the following result.

\begin{prop}\label{prop321}   Let    $\Lambda$ be a virtually
Gorenstein Artin algebra. Then  {\rm (GPC) $\Rightarrow$ (ARC)}. \end{prop}

 \begin{proof} Assume  $M \in$  {\rm mod}-$\Lambda$   satisfies   {\rm Ext}$_{\Lambda}^{i}(M, M\oplus\Lambda) =  0$ for all  $i \geq 1$.
 Then, since  $\Lambda$  is virtually
Gorenstein, we see from Theorem \ref{thm101}  that $M  $ is Gorenstein projective. Hence, the  validity of   (GPC)  implies that $M  $ is   projective, as desired.   \end{proof}

  Zhang showed in \cite[Theorem 3.6]{luh} that  (GPC) holds for   Artin algebras   of finite CM-type. Building on this result and Proposition \ref{prop321}, we obtain the following new conclusion.
  \begin{thm}\label{thm501} The Auslander-Reiten Conjecture holds for all virtually Gorenstein  Artin algebras  of finite CM-type.
 \end{thm}

\begin{proof} Let $\Lambda$ be a virtually
Gorenstein Artin algebra  of finite CM-type. Then (GPC) holds for $\Lambda$  by  \cite[Theorem 3.6]{luh}. Therefore, (ARC) holds for $\Lambda$  by
Proposition \ref{prop321}. \end{proof}

 The following theorem provides
another proof of (GSC) for virtually Gorenstein Artin algebras (see \cite[Theorem  11.4]{bel}).

\begin{thm}\label{thm102}  Let  $\Lambda$ be a virtually Gorenstein Artin algebra and  $ E_{i} $  $(i \geq 0)$ be the module as in  {\rm Notation \ref{nota1}}. Then the following quantities are identical.
\begin{itemize}
\item[$(1)$]  {\rm id$_{\Lambda}$($\Lambda$)}.
 \item[$(2)$] {\rm id$_{\Lambda^{{\rm op}}}$($\Lambda^{{\rm op}}$)}.
 \item[$(3)$]  {\rm inf\{$n \in \mathbb{N}$\ $|$\   Ext$_{\Lambda}^{n+l}(D(\Lambda^{{\rm op}}), \Lambda) =  0$ for all $l\geq 1$\}}.
 \item[$(4)$]  {\rm inf\{$n \in \mathbb{N}$\ $|$\   Ext$_{\Lambda^{{\rm op}}}^{n+l}(D(\Lambda), \Lambda^{{\rm op}}) =  0$ for all $l\geq 1$\}}.
  \item[$(5)$]  {\rm sup\{pd$_{\Lambda}$$(E_{i})$\ $|$\ $i \in \mathbb{N}$\}}.\end{itemize}\end{thm}

 \begin{proof} We get from Theorem \ref{thm91} that {\rm Gid$_{\Lambda}$($\Lambda$)} =   (3).
 But {\rm Gid$_{\Lambda}$($\Lambda$)} = {\rm id$_{\Lambda}$($\Lambda$)} per \cite[Theorem  2.1]{hol1}. So (1) $= (3)$.
Similarly, we have (2) $= (4)$.

Dually, we obtain (3) =   pd$_{\Lambda}$$(D(\Lambda^{{\rm op}})$) by using Theorem \ref{thm92} and  \cite[Theorem  2.2]{hol1}.
 Note that $D(\Lambda^{{\rm op}})$ is an injective cogenerator, and each $E_{i}$ is finitely generated.
 Thus $E_{i} \in$ add$(D(\Lambda^{{\rm op}}))$. These observations imply that (5) $\leq (3)$.

 Next we prove the equality (3) $= (4)$. Note that $DD(\Lambda) \cong \Lambda$  and $DD(\Lambda^{{\rm op}}) \cong \Lambda^{{\rm op}}$. Hence
${\rm Ext}_{\Lambda}^{n}(D(\Lambda^{{\rm op}}), \Lambda) \cong {\rm Ext}_{\Lambda}^{n}(D(\Lambda^{{\rm op}}), DD(\Lambda)) \cong
 D{\rm Tor}_{n}^{\Lambda}(D(\Lambda^{{\rm op}}), D(\Lambda))
 \cong      {\rm Ext}_{\Lambda^{{\rm op}}}^{n}(D(\Lambda), DD(\Lambda^{{\rm op}})) \cong      {\rm Ext}_{\Lambda^{{\rm op}}}^{n}(D(\Lambda), \Lambda^{{\rm op}})$
by  \cite[ Lemma 2.16(b)]{gbe}. This shows that (3) $= (4)$.

 It remains to argue that (3) $\leq (5)$.  We may assume {\rm sup\{pd$_{\Lambda}$$(E_{i})$\ $|$\ $i \in \mathbb{N}$\}} = $t < \infty$.
  Let    $\mathscr{S}$  be an irredundant finite set of representatives of the simple modules in Mod-$\Lambda$.
Write $G = \bigoplus_{S\in \mathscr{S} } E(S)$, where $E(S)$   is the injective envelope of    $S$.
 For any  $S\in$ $\mathscr{S}$,  we get some  $n \geq 0$ such that  ${\rm Ext}_{\Lambda}^{n}(S, \Lambda)\neq 0$
by Corollary \ref{cor10}.
 So $E(S)$ is a direct summand of the module $E_{n}$ in Notation \ref{nota1} (cf. \cite[p. 70]{ar0}), and hence {\rm pd}$_{\Lambda}$($E(S)) \leq  t$.
 Thus   pd$_{\Lambda}$$(G) \leq t$. But  $G $ is an injective cogenerator. Therefore,
  {\rm pd}$_{\Lambda}$($D(\Lambda^{{\rm op}}))$  $\leq t$, which implies   that (3) $\leq (5)$. The proof is finished.
 \end{proof}

\begin{rem}\label{rem30}  {\rm
The validity of (TC1) in the setting of virtually Gorenstein Artin algebras
can be directly deduced as a special instance of Theorem \ref{thm102}.}\end{rem}

 In view of the above discussions, it is natural to close the present paper by putting forward the following conjecture,
  which can be regarded as a   generalized extension covering   (AGC) and (TC1).
  \\
\\
{\bf Conjecture 3.20.} {\it Let  $\Lambda$ be an Artin algebra   and  $ E_{i} $  $(i \geq 0)$ be the module as in  {\rm Notation \ref{nota1}}.
 Then the following quantities are identical.
\begin{itemize}
\item[$(1)$]  {\rm id$_{\Lambda}$($\Lambda$)}.
 \item[$(2)$] {\rm inf\{$n \in \mathbb{N}$\ $|$\   Ext$_{\Lambda}^{n+l}(D(\Lambda^{{\rm op}}), \Lambda) =  0$ for all $l\geq 1$\}}.
 \item[$(3)$]  {\rm sup\{pd$_{\Lambda}$$(E_{i})$\ $|$\ $i \in \mathbb{N}$\}}.\end{itemize}}

\section*{Acknowledgments}

The author would like to thank Professor Hongxing Chen and his colleagues, as well as his students, for their contributions to this paper, which have significantly enhanced the
 quality of this work. The author is also grateful to Professor Pu Zhang  for his  helpful suggestions on this paper.

\section*{Disclosure statement}
The author reports there are no competing interests to declare. This note has no associated data.


\vskip3mm

\begin{thebibliography}{99}

\bibitem{am} J. Asensio Mayor, J. Martinez Hernandez, On flat and projective envelopes,   {\it J. Algebra } 160 (1993),  434-440.

\bibitem{and} F. W. Anderson,  K. R. Fuller,  {\it Rings and
Categories  of  Modules}, 2nd. Springer-Verlag, New York, 1992.

\bibitem{ar0}  M. Auslander, I. Reiten, On a generalized version of Nakayama conjecture, {\it Proc. Amer. Math. Soc.}  52 (1975), 69-74.

\bibitem{ar1}  M. Auslander, I. Reiten, Applications of contravariantly finite subcategories, {\it Adv. Math.} 86 (1991), 111-152.

\bibitem{ar2}  M. Auslander, I. Reiten,  $k$-Gorenstein algebras and syzygy modules, {\it J. Pure Appl.
Algebra} 92 (1994), 1-27.

\bibitem{ars} M. Auslander, I. Reiten, S. O. Smal${\o}$, {\it Representation Theory of Artin Algebras.} Cambridge Studies in Advanced Mathematics, vol. 36. Cambridge Univ.  Press, Cambridge, 1995.

\bibitem{bel} A. Beligiannis, Cohen-Macaulay modules, (co)torsion pairs and virtually Gorenstein algebras, {\it J. Algebra }  288 (2005), 137-211.

\bibitem{bel1} A. Beligiannis, On algebras of finite Cohen-Macaulay type,  {\it Adv. Math. }  226 (2011), 1973-2019.

\bibitem{bek} A. Beligiannis, H. Krause,  Thick subcategories and virtually gorenstein algebras,  {\it Ill. J. Math. }  52(2)  (2008), 551-562.

\bibitem{bm} B. B$\ddot{\rm o}$hmler, R. Marczinzik, Tor and Ext vanishing results for commutative Artinian rings,  arXiv:2608.09701v1.

\bibitem{ber} A. Beligiannis, I. Reiten,  Homological and homotopical aspects of torsion theories, {\it Mem. Amer. Math. Soc}. 188(883) (2007), 1-207.

\bibitem{cum} C. Cummings, Left-right symmetry of finite finitistic dimension, {\it Bull. Lond. Math. Soc.} 56(2) (2024),  624-633.


\bibitem{cfx} H. X. Chen, M. Fang, C. C. Xi, Tachikawa's second conjecture, derived recollements, and gendo-symmetric algebras, {\it Compos. Math.} 160(11),  (2024), 2704-2734.

\bibitem{cof} R. R. Colby, K. R. Fuller,  A note on the Nakayama Conjecture, {\it Tsukuba J. Math.}
14 (1990), 343-352.

\bibitem{chh} L. W. Christensen, H. Holm,  Algebras that satisfy Auslander's condition on vanishing
of cohomology, {\it Math. Z.} 265 (2010),  21-40.

\bibitem{chx} H. X. Chen, C. C. Xi,  Homological theory of self-orthogonal modules,  {\it Trans. Amer. Math. Soc.} 378(10) (2025),  7287-7335.

\bibitem{ch1} H. X. Chen,   C. C. Xi,  Virtually Gorenstein algebras of infinite dominant dimension, {\it J. Pure Appl.
Algebra} 230(4) (2026), 108224.

\bibitem{eno9}  E. E. Enochs, O. M. G. Jenda, Gorenstein injective and projective modules,  {\it Math. Z}. 220(4)  (1995), 611-633.

\bibitem{eno5} E. E. Enochs, O. M. G. Jenda, {\it Relative  Homological  Algebra},  Berlin,  Walter de Gruyter,
2000.


\bibitem{gii} J. Gillespie, A. Iacob, Duality pairs, generalized Gorenstein modules, and Ding injective
envelopes, {\it Comptes Rendus Math$\acute{{\rm e}}$matique}  360 (2022), 381-398.

\bibitem{gbe} R. G$\ddot{\rm o}$bel, J. Trlifaj, {\it Approximations and Endomorphism Algebras of Modules}, de Gruyter Exp. Math. vol.
41,  2 revised and extended edn. de Gruyter, Berlin, (2012).

\bibitem{hol1} H. Holm,  Rings with finite gorenstein injective dimension, {\it Proc. Amer. Math. Soc.} 132(5) (2004), 1279-1283.

\bibitem{hol} H. Holm,  Gorenstein homological dimensions. {\it J. Pure Appl. Algebra} 189 (2004), 167-193.

\bibitem{hua} Z. Y. Huang, Selforthogonal modules with finite injective dimension II, {\it J. Algebra} 264 (2003),   262-268.

\bibitem{kim} K. Kimura, Y. Mifune, Y. Otake, R. Takahashi, On strongly $G$-regular rings, arXiv:2608.08228v1.

\bibitem{li} W. Q. Li, On $G$-$(n,d)$-rings and $n$-coherent rings, {\it  Int. Electron. J. Algebra} 37 (2025), 147-178.

\bibitem{luh} R. Luo,  Z. Y. Huang,  When are torsionless modules projective? {\it J.   Algebra} 320(5), (2008), 2156-2164.

\bibitem{mao}  L. X. Mao,  N. Q. Ding,  Relative cotorsion modules and relative flat modules, {\it Comm. Algebra} 34 (2006), 2303-2317.

\bibitem{mue} B. J.  M$\ddot{\rm u}$eller, The classification of algebras by dominant dimension, {\it Canad. J. Math.} 20 (1968), 398-409.

 \bibitem{poj} P. Moradifar, J. $\check{\rm S}$aroch, J. $\check{\rm S}$$\check{\rm t}$ov$\acute{\rm {\i}}$$\check{\rm c}$ek,   Finitistic dimension conjectures via Gorenstein
projective dimension, {\it J. Algebra} 591 (2022), 15-35.

\bibitem{nak} T. Nakayama, On algebras with complete homology, {\it Abh. Math. Sem. Univ. Hamburg} 22  (1958), 300-307.

\bibitem{rz1} C. M. Ringel, P. Zhang, Gorenstein-projective and semi-Gorenstein-projective modules, {\it Algebra $\&$   Number Theory}
14(1) (2020), 1-36.

\bibitem{riz}  C. M. Ringel, P. Zhang, Gorenstein-projective modules over short local algebras. {\it J. Lond. Math. Soc.} 106(2) (2022), 528-589.

 \bibitem{ss} J. $\check{\rm S}$aroch, J. $\check{\rm S}$$\check{\rm t}$ov$\acute{\rm {\i}}$$\check{\rm c}$ek,   Singular compactness and definability for $\sum$-cotorsion and Gorenstein modules,

{\it Sel. Math. New Ser.} 26:23 (2020).

\bibitem{tac} H. Tachikawa, {\it Quasi-Frobenius rings and generalizations:  QF-{\rm 3} and QF-{\rm 1} rings}, Notes by Claus Michael Ringel. Lecture Notes in Mathematics, 351, Berlin-Heidelberg New York, 1973.

\bibitem{zh}   X. J. Zhang, A note on Gorenstein projective conjecture II, {\it Nanjing Daxue Xuebao Shuxue Bannian
Kan} 29(2) (2012), 155-162.
\end{thebibliography}
\end{document}